\documentclass[10pt]{article}
\usepackage{amssymb,amsmath, amsthm, amsfonts}
\usepackage{mathrsfs,color}
\usepackage[all]{xypic}
\usepackage{graphicx}
\usepackage{tikz,color} 
\usetikzlibrary{calc,matrix,arrows,chains,positioning,scopes,decorations.pathmorphing,backgrounds,fit}

\usepackage{yfonts}
\newcommand{\NL}{$\mathcal{IR}$\L}
\usepackage{enumitem}
\usepackage{adjustbox}
\usepackage{thmtools,mathtools}

\usepackage{nameref,cleveref}

\newtheorem{theorem}{Theorem}[section]
\newtheorem{lemma}[theorem]{Lemma}
\newtheorem{corollary}[theorem]{Corollary}
\newtheorem{proposition}[theorem]{Proposition}

\newtheorem{notation}{Notation}

\theoremstyle{definition}
\newtheorem{definition}[theorem]{Definition}
\newtheorem{remark}[theorem]{Remark}

\tikzset{node distance=2.5cm, auto}
\newcommand{\quot}[2]{{\raisebox{.2em}{$#1$}\left/\raisebox{-.2em}{$#2$}\right.}}

\begin{document}

\title{Infinitary logic and basically disconnected compact Hausdorff spaces}
\author{Antonio Di Nola,
Serafina Lapenta\footnote{Corresponding author}\\ 
{\small Department of Mathematics, University of Salerno,}\\ 
{\small  Via Giovanni Paolo II, 132 Fisciano (SA), Italy}\\
 {\small adinola@unisa.it, slapenta@unisa.it}
 \and
Ioana Leu\c stean \\
{\small Department of Computer Science,} \\
{\small Faculty of Mathematics and Computer Science, University of Bucharest,}\\
{\small Academiei nr.14, sector 1, C.P. 010014,  Bucharest, Romania}\\   
{\small ioana@fmi.unibuc.ro}
}
\date{}

\maketitle

\begin{abstract}
We extend \L ukasiewicz logic obtaining the infinitary logic \NL \ whose models are algebras $C(X,[0,1])$, where $X$ is a basically disconnected compact Hausdorff space. Equivalently, our models are unit intervals in Dedekind $\sigma$-complete Riesz spaces with strong unit. The Lindenbaum-Tarski algebra of  \NL \  is, up to isomorphism, an algebra of $[0,1]$-valued Borel functions. Finally, our system enjoys standard completeness with respect to the real interval $[0,1]$.


\end{abstract}

\section{Introduction}

Riesz Spaces have had a predominant position in the development of functional analysis over ordered structures. The monographs \cite{RS, RSZan} are a good reference to grasp the widespread of applications and results related to Riesz spaces (vector lattices) and Banach lattices.

Not very known is the interplay between the theory of  Riesz spaces and \L ukasiewicz logic: given any positive element $u$ of a Riesz space $V$, the interval $[0,u]=\{x\in V\mid 0\le x\le u\}$ can be endowed with a stucture of {\em Riesz MV-algebra} \cite{LeuRMV}. These structures are expansions of MV-algebras, the standard semantics of the 
$\infty$-valued \L ukasiewicz logic \cite{CDM, MunBook}.  The well known categorical equivalence between MV-algebras and lattice-ordered groups with strong unit \cite{Mun1} led to a categorical equivalence between Riesz MV-algebras and Riesz spaces with strong unit. Consequently, one can develop a logical system that extends \L ukasiewicz logic and has Riesz MV-algebras as models \cite{LeuRMV}.
 
The strong connection between functional analysis and Riesz spaces should be reflected by their underlying logical systems and in the present paper  we test the expressive power of the ``logic of Riesz Spaces"  in order to provide a logical system that is closer to what we may think of as the ``logic of (some) Hausdorff spaces". Remarkably, by adding a countable disjunction to the logic of Riesz MV-algebras we were able to obtain the desired bridge between logic, topology  and functional analysis.

Our construction rests on Kakutani's duality between abstract M-spaces and compact Hausdorff spaces \cite{kakutani}. Indeed, from a categorical point of view, the \emph{norm-complete} Riesz MV-algebras\footnote{See \Cref{sec:pre} for the definition of a norm-complete Riesz MV-algebra.} are equivalent to M-spaces, and henceforth dual to compact Hausdorff spaces \cite{LeuRMV}.  Our aim is to express the property of being ``norm-complete" in a logical setting.  Thus, in \cite{DiNLL-RMV} we introduced the limit of a sequence of formulas, a syntactic notion whose semantic counterpart is the uniform limit of the corresponding sequence of term functions.  This notion is slightly stronger than the usual  notion of order convergence and this remark has been the starting point of the present development.

In Section \ref{sec:logic} we define the system \NL , that stands for \textit{Infinitary Riesz Logic},  which expands the logic of Riesz spaces with denumerable disjunctions and conjunctions. 
In this way we define a logical system whose models are isomorphic to unit intervals of {\em $\sigma$-complete M-spaces} and, consequently, strongly related to  \textit{basically disconnected compact Hausdorff spaces}.  The Lindenbaum-Tarski algebra in $n$ variables is concretely characterized as the algebra of all Borel functions $f:[0,1]^n\to [0,1]$. Moreover, our system enjoys standard completeness: the real interval $[0,1]$ endowed with a structure of $\sigma$-complete Riesz MV-algebra is a standard model.

The present approach via infinitary logic is built upon the work of C.R. Karp \--- of which the monograph \cite{Karp} is a complete treatise \--- where we replace the classical axioms of Boolean logic with the axioms of Riesz \L ukasiewicz logic. 

After the needed preliminaries, we define the logical system \NL \ in Section \ref{sec:logic}, where we also prove a general completeness theorem. 
The link between the models of \NL \ and Kakutani's duality is provided in \Cref{sec:KHSpace}. In the last section we prove the Loomis-Sikorski theorem for Riesz MV-algebras, based on the well-known similar result for MV-algebras \cite{Dvu-LS,Mun}. Consequently, we prove the standard completeness of \NL \ and 
we give a concrete characterization of the free $n$-generated $\sigma$-complete Riesz MV-algebra in terms of $[0,1]$-valued functions defined over $[0,1]^n$.
  
\section{Preliminaries}\label{sec:pre}
In the first part of this section we recall the interaction between \L ukasiewicz logic and the theory of Riesz spaces, while the second part is a short presentation of the infinitary classical logic. 

\subsection{\L ukasiewicz logic and Riesz MV-algebras}\label{sec:prelimRMV}
\noindent  An {\em MV-algebra} is a structure $(A,\oplus,^{*},0)$ of type $(2,1,0)$ such that, for any $x,y\in A$,
$$\begin{array}{ll}
 (A,\oplus,0) \mbox{ is an Abelian monoid}, & ({x^*})^* = x\\
 (x^* \oplus y)^* \oplus y = (y^* \oplus x)^* \oplus x, & 0^* \oplus x = 0^*.
	\end{array}$$
They were first introduced as an algebraic counterpart of the propositional \L ukasiewicz infinite-valued logic, but they soon gained a relevant place in the theory of lattice-ordered algebraic structures.

On an MV-algebra one can define further operations as follows: $1$ is $ 0^*$, the  \L ukasiewicz
implication is $x\to y=x^* \oplus y$ and the \L
ukasiewicz conjunction is $x\odot y =( x^*\oplus
y^*)^*$ for any $x$, $y\in A$. If $x\vee y = x\oplus (y\odot x^*)$ and $x\wedge y = (x^*\vee y^*)^*$ then $(A,\vee,\wedge, 1,0)$ is a bounded distributive lattice. 

It is straightforward that MV-algebras form a variety, in which the \textit{standard model}, that is the generator of the variety, is
$$ [0,1]_{MV}=([0,1], \oplus, ^*,0),$$
 where $[0,1]$ is the real unit interval,  $x^*=1-x$ and  $x\oplus y=\min(1,x+y)$ for any $x,y\in [0,1]$.  We urge the interested reader to consult  \cite{CDM,DinLeu}  for a basic introduction to MV-algebras and \L ukasiewicz logic, and \cite{MunBook} for advanced topics.

A fruitful research direction has arisen from the idea of endowing MV-algebras with a product operation. In particular, when we consider a scalar multiplication, we obtain the notion of a \textit{Riesz MV-algebra}.

A {\em Riesz MV-algebra} \cite{LeuRMV} is a structure 
$$(R, \oplus, ^*, 0, \{r\mid r \in [0,1]\}),$$ where $(R, \oplus, ^*, 0)$ is an MV-algebra and $\{r\mid r \in [0,1]\}$ is a family of unary operations such that  the following  properties hold for any $x,y \in A$ and $r,q\in [0,1]$:
$$\begin{array}{ll}
r (x\odot y^{*})=(r  x)\odot(r y)^{*},& 
(r\odot q^{*})\cdot x=(r x)\odot(qx)^{*},\\
r (q  x)=(rq) x,&
1 x=x. 
\end{array}$$

As for MV-algebras, the unit interval provides the standard model of a Riesz MV-algebra, indeed the variety of Riesz MV-algebras is generated by 
 $$[0,1]_{\mathbb R}=([0,1],\oplus, ^*, \{\alpha\mid \alpha\in [0,1]\}, 0),$$
 where $([0,1],\oplus, ^*,0)$ is the standard MV-algebra and  $x\mapsto \alpha x$ is a unary operation for any $\alpha\in [0,1]$ that is interpreted in the product of real numbers.

Both MV-algebras and Riesz MV-algebras can be approached from the point of view of category theory: $\mathbf{MV}$ and $\mathbf{RMV}$ denote the category whose objects are MV-algebras and Riesz MV-algebras respectively, and whose arrows are homomorphisms of MV-algebras and Riesz MV-algebras respectively. \textit{Semisimple} MV-algebras, which play a central role in the connections with logic, are subdirect products of  MV-subalgebras of $[0,1]$ or, equivalently, they are sub-MV-algebras of $C(X)$, the continuous functions from a compact Hausdorff space $X$ to $[0,1]$. A  Riesz MV-algebra is semisimple if its MV-algebra reduct is semisimple. 

For what follows, a particularly important full subcategory of Riesz MV-algebras is the one of \textit{norm-complete} Riesz MV-algebras \cite{LeuRMV}. In any  Riesz MV-algebra $R$ it is possible to define the {\em unit seminorm} $\|\cdot\|_u:R\to [0,1]$  by $\|x\|_u=\min\{r\in [0,1]\mid  x\leq r1_R\}$ for any $x\in R$. Such a seminorm induces a pseudometric $\rho_{\|\cdot\|_u}$; when $R$ is semisimple, $\|\cdot\|_u$ is a norm and $\rho_{\|\cdot\|_u}$ is a metric. We say that a Riesz MV-algebra $R$  is {\em norm-complete} if $(R,\rho_{\|\cdot\|_u})$ is a complete metric space. Any norm-complete Riesz MV-algebra $R$ is isometrically isomorphic with $C(X)=\{f:X\to [0,1]\mid f \mbox{ continuous}\}$, where $X$ coincides with $Max(R)$ \--- the space of all maximal ideals of $R$ \--- and $\|f\|_u= \sup\{f(x)\mid x\in X\}$. 

If $RMV_n$ denotes the free $n$-generated Riesz MV-algebra, in \cite[Corollary 7]{LeuRMV} it is proved that $RMV_n$ is a Riesz MV-algebra of appropriate\footnote{$RMV_n$ is isomorphic with the algebra of piecewise linear functions with \emph{real} coefficients, defined over $[0,1]^n$. For further details and a proper definition see \cite{LeuRMV}.} $[0,1]$-valued functions defined on $[0,1]^n$.  If we denote by $RL$ the Lindenbaum-Tarski algebra of the logic of Riesz MV-algebras $\mathcal{R}\L$  and by $RL_n$ the Lindenbaum-Tarski algebra in $n$ propositional variables, we can prove that $RL_n$ is isomorphic to $RMV_n$. In \cite[Theorem 2.15]{DiNLL-RMV} it is proved that the norm-completion of $RL_n$ coincides with $C([0,1]^n)$, where both algebras are endowed with the unit norm defined above.
 
Finally, we recall that the categories $\mathbf{MV}$ and $\mathbf{RMV}$ are equivalent to the ones of lattice-ordered groups with strong unit ($\ell u$-groups) and Riesz spaces (vector lattices) with strong unit respectively. The functors that give the equivalence, denoted by $\Gamma$ and $\Gamma_\mathbb{R}$, are defined similarly to each other: in the case of MV-algebras, for any $\ell u $-group $(G,u)$, $\Gamma(G,u)=\{ x\in G \mid 0\le x\le u\}=[0,u]_G$  and for any morphism $f\colon (G,u)\to (H,v)$, $\Gamma (f)=f\mid_{[0,u]_G}$. The following result will be useful subsequently.
\begin{proposition}\label{Gamma-DedCompl}
$\Gamma$ and $\Gamma_\mathbb{R}$ can be restricted and co-restricted to the subcategories that have Dedekind $\sigma$-complete objects and $\bigvee$-preserving morphisms.
\end{proposition}
\begin{proof}
Since the property of being Dedekind $\sigma$-complete only depends on the lattice reduct of the algebra, we will prove the result for MV-algebras and $\ell u$-groups, but the same applies to Riesz MV-algebras.

The fact that $(G,u)$ is Dedekind $\sigma$-complete if, and only if, $[0,u]_G$ is $\sigma$-complete is well known, see for example \cite[Theorem 16.9]{goodearl}. We need to prove that $\Gamma (f)=f^*$ is $\bigvee$-preserving if, and only if, $f$ is so. For the non trivial direction, we recall that $(G,u)$ is isomorphic with the group of \emph{good sequences} built upon $[0,u]_G$. For such a group, we will use the simplified construction from \cite{BGL}. We have to prove that $f(\bigvee_n a_n)=\bigvee_n f(a_n)$, for $a_n\in G$. By \cite{BGL}, for any $a_n$ there exists a good sequence $(b_m^n)_{m\in\mathbb{Z}}$ that corresponds to $a_n$. Thus,
\begin{align*}
f\left(\bigvee_n a_n\right)=&f\left(\bigvee_n (\dots, b_1^n, b_2^n, \dots b_m^n, \dots)\right)=\\
=&f\left(\left(\dots, \bigvee_n b_1^n, \bigvee_n b_2^n, \dots, \bigvee_n b_n^n, \dots\right)\right)=\\
=&\left(\dots, f^*\left(\bigvee_n b_1^n\right), f^*\left(\bigvee_n b_2^n\right), \dots, f^*\left(\bigvee_n b_n^n\right), \dots\right)=\\
=&\left(\dots, \bigvee_n f^*(b_1^n), \bigvee_n f^*(b_2^n), \dots, \bigvee_n f^*(b_n^n), \dots\right)=\\
=&\bigvee_n \left(\dots, f^*(b_1^n), f^*(b_2^n), \dots,  f^*(b_n^n), \dots\right)=\\
=&\bigvee_n f(\dots, b_1^n, b_2^n, \dots b_m^n, \dots)= \bigvee_n f(a_n).
\end{align*}
\end{proof}

\subsection{Infinitary classical propositional logic}\label{ksec}
The idea of endowing classical propositional logic with infinitely long sentences is relatively new, and the first published results are due to C.R. Karp \cite{Karp} and D. Scott and A. Tarski \cite{ScottTarski} in the beginning of the Sixties. In both cases, the authors treat a much more general case than what is needed in the current investigation. We briefly recall it here.

Given $\kappa$ an infinite cardinal, the core idea is to extend classical propositional logic by defining a language in which one can build conjunctions of sets of formulas of cardinality $\alpha < \kappa$.

A complete treatise of the subject is \cite{Karp}, in which the logical system $\mathcal{B}_{\kappa}$ is defined  as the system whose connectives are $\rightarrow, \neg, \bigwedge$ and whose axioms are the following:
\begin{enumerate}[label=(IL\arabic*)]
\item \label{def:IL:item1} $\varphi\rightarrow(\psi\rightarrow \varphi)$
\item \label{def:IL:item2} $(\varphi\rightarrow(\psi\rightarrow \chi))\rightarrow((\varphi\rightarrow\psi)\rightarrow(\varphi \rightarrow\chi))$
\item \label{def:IL:item3} $(\neg \varphi \rightarrow\neg \psi)\rightarrow(\psi \rightarrow \varphi)$
\item \label{def:IL:item4} $(\bigwedge_{\eta \le \alpha} (\varphi_{\alpha}\rightarrow \varphi_{\eta}))\rightarrow (\varphi_{\alpha}\rightarrow\bigwedge_{\eta\le \alpha}\varphi_{\eta})$, for $0< \alpha <\kappa$
\item \label{def:IL:item5} $\bigwedge_{\eta\le \alpha}\varphi_{\eta}\rightarrow \varphi_{\nu}$, where $\nu<\alpha$.
\end{enumerate}
Deduction rules are Modus Ponens and the following

$$\text{(INF) \quad }\frac{\varphi_1, \dots \varphi_{\alpha}}{\bigwedge_{\eta\le \alpha}\varphi_{\eta}} $$

The system we present in Section \ref{sec:logic} has a slightly different axiomatization, but it is heavily inspired by $\mathcal{B}_{\kappa}$.

\section{The logic \NL}\label{sec:logic}

As recalled in Section \ref{sec:prelimRMV}, it is possible to obtain a conservative extension of \L ukasiewicz logic by expanding the language with a collection of connectives $ \nabla_\alpha$, with $\alpha \in [0,1]$. In this section we will build an infinitary system, based on \cite{Karp}, that will allow us to obtain a logic whose models are spaces of continuous functions.

Let us consider a countable set of propositional variables and the connectives $\neg$, $\rightarrow$, $\{ \nabla_\alpha\}_{\alpha\in [0,1]}$, $\bigvee$. The connectives $\neg$, $\rightarrow$, $\{ \nabla_\alpha\}_{\alpha\in [0,1]}$ are the correspondent ones from the logic $\mathcal{R}$\L, while the latter is a connective of arity less than or equal to $\omega$, i.e. it is defined for any set of formulas which is at most countable. Consider now the following set of axioms:
\begin{enumerate}[label=(L\arabic*)]
\item \label{L:item1}  $\varphi \rightarrow (\psi \rightarrow \varphi)$

\item \label{L:item2} $(\varphi \rightarrow \psi)\rightarrow((\psi \rightarrow \chi)\rightarrow(\varphi \rightarrow \chi))$

\item \label{L:item3} $((\varphi \rightarrow \psi)\rightarrow \psi)\rightarrow ((\psi \rightarrow \varphi)\rightarrow \varphi)$

\item \label{L:item4} $(\neg \psi \rightarrow \neg \varphi)\rightarrow (\varphi \rightarrow \psi)$
\end{enumerate}
\begin{enumerate}[label=(R\arabic*)]
\item \label{R:item1}  $\nabla_{\alpha}(\varphi \rightarrow \psi)\leftrightarrow (\nabla_{\alpha}\varphi \rightarrow \nabla_{\alpha}\psi)$
\item \label{R:item2} $ \nabla_{(\alpha \odot \beta ^*)}\varphi \leftrightarrow (\nabla_{\beta}\varphi \rightarrow \nabla_{\alpha}\varphi)$

\item \label{R:item3} $\nabla_{\alpha}(\nabla_{\beta}\varphi)\leftrightarrow \nabla_{\alpha \cdot \beta} \varphi$
\item \label{R:item4} $\nabla_1 \varphi \leftrightarrow \varphi$
\end{enumerate}
\begin{enumerate}[label=(S\arabic*)]
\item \label{S:item1}  $\varphi_k  \rightarrow \bigvee_{n\in \mathbb{N}}\varphi_n $, for any $k\in \mathbb{N}$.
\end{enumerate}

\Cref{L:item1,L:item2,L:item3,L:item4} are the axioms of \L ukasiewicz logic, \cref{R:item1,R:item3,R:item2,R:item4} are the additional axioms of the logic of Riesz MV-algebras, the deduction rules will be the \textit{Modus Ponens} and the following:
$$
\text{(SUP) \quad }\frac{(\varphi_1 \rightarrow \psi),  \dots, (\varphi_k\rightarrow \psi) \dots}{\bigvee_{n\in \mathbb{N}}\varphi_n \rightarrow \psi}
$$
Theorems, deductions, proofs and so on are defined as usual. Moreover, we can define the derivative connective $\bigwedge_{n\in \mathbb{N}}\varphi_n$ as $\neg \bigvee_{n\in \mathbb{N}} \neg \varphi_n$. 

As we shall prove subsequently (see \Cref{rem:sost-equiv} and \Cref{pro:consext}) our system is a conservative extension of the logic of Riesz MV-algebras.

We should have defined axiom (S1) and (SUP) taking into account the possibility of applying the connective $\bigvee$ to a set of formulas of cardinality $< \omega$. It is understood in what follows that we allow such cases by setting $\bigvee\{\varphi, \psi \}$ to coincide with $\bigvee\{ \varphi, \psi, \bot, \bot, \dots \}$ for countably many $\bot$.

\begin{notation}
Let $\varphi$ and $\psi$ be two formulas of \NL . We shall write $\varphi \le \psi$ whenever $\vdash \varphi \rightarrow \psi$. Moreover, we will call a sequence of formulas  $\{ \varphi_n \}_{n\in \mathbb{N}}\subseteq FORM_{IRL}$ increasing (decreasing) if $\varphi_k\le \varphi_{k+1}$ ($\varphi_k\le \varphi_{k-1}$).
\end{notation}

\begin{remark}
In \L ukasiewicz logic, the binary disjunction $\varphi \vee \psi$ is defined as an abbreviation for $(\varphi\rightarrow\psi)\rightarrow\psi$. In the following proposition we will see, among other properties, that $\varphi \vee \psi$ is equivalent to $\bigvee \{\varphi, \psi \}$.
\end{remark}

\begin{proposition}\label{pro:basic}
The following properties hold: 
\begin{enumerate}[label=(\roman*)]
\item \label{pro:basic:item1} Let $\{\varphi_1, \varphi_2, \dots , \varphi_n, \dots \}$ and $\{\psi_1, \psi_2, \dots , \psi_n, \dots \}$ be two sequences of formulas in \NL . If $\vdash \varphi_k \rightarrow \psi_k$ for any $k$, then $\vdash \bigvee_{n\in \mathbb{N}}\varphi_n \rightarrow \bigvee_{n\in \mathbb{N}}\psi_n $;
\item \label{pro:basic:item6} (Idempotency) If $\{\varphi_1, \varphi_2, \dots , \varphi_n, \dots \}$ is a sequence such that $\vdash \varphi_k\leftrightarrow \varphi$ for any $k\in \mathbb{N}$, then $\vdash \bigvee_{n\in \mathbb{N}}\varphi_n\leftrightarrow \varphi$.
\item \label{pro:basic:item2} Let $\varphi, \psi$ be formulas in \NL, then $\vdash (\varphi \vee \psi) \leftrightarrow \bigvee \{\varphi, \psi\}$;
\item \label{pro:basic:item3}The following derivative rule holds: $\displaystyle{(INF) \quad\frac{\psi\rightarrow \varphi_1, \dots, \psi \rightarrow \varphi_k}{\psi \rightarrow \bigwedge_{n \in \mathbb{N}}\varphi_n}}$.
\item \label{pro:basic:item4} $\vdash \bigwedge_{n\in \mathbb{N}} \varphi_n \rightarrow \varphi_k$, for any $k\in \mathbb{N}$.
\item \label{pro:basic:item5} (Distributivity of $\odot$) $\vdash \left(\varphi  \odot \bigvee_{n\in \mathbb{N}} \psi_n\right)\leftrightarrow \bigvee_{n\in \mathbb{N}}\left( \varphi\odot \psi_n\right)$, where $\varphi\odot \psi$ is defined as $\neg( \varphi\to \neg\psi)$.
\item \label{pro:basic:item7} (Distributivity of $\rightarrow$) $\vdash \left(\varphi  \rightarrow \bigwedge_{n\in \mathbb{N}} \psi_n\right)\leftrightarrow \bigwedge_{n\in \mathbb{N}}\left( \varphi\rightarrow \psi_n\right)$.
\end{enumerate}
\end{proposition}
\begin{proof}
\ref{pro:basic:item1} By hypothesis, we have $\vdash \varphi_k \rightarrow \psi_k$. Combined with axioms \ref{S:item1} and \ref{L:item2}, it implies that $\vdash \varphi_k \rightarrow \bigvee_{n\in \mathbb{N}}\psi_n$ for any $k\in \mathbb{N}$. Hence, by (SUP) $\vdash \bigvee_{n\in \mathbb{N}}\varphi_n \rightarrow \bigvee_{n\in \mathbb{N}}\psi_n $. 

\vspace*{0.2cm}
\ref{pro:basic:item6} By hypothesis, $\vdash \varphi_k\rightarrow \varphi$ and $\vdash \varphi\rightarrow \varphi_k$ for any natural $k$. Thus, from the first theorem and (SUP) we deduce  $\vdash \bigvee_{n\in \mathbb{N}}\varphi_n\rightarrow \varphi $, while from the second theorem and \ref{S:item1} we infer $\vdash \varphi \rightarrow \bigvee_{n\in \mathbb{N}}\varphi_n$, which altogether settle the claim.

\vspace*{0.2cm}
\ref{pro:basic:item2} By \ref{S:item1} we derive $\vdash \varphi \rightarrow \bigvee \{\varphi, \psi\}$ and $\vdash \psi \rightarrow\bigvee \{\varphi, \psi\}$. By the ``proof by cases" theorem \cite[Theorem 1.2.11]{Chapter1-HB} we derive $\vdash (\varphi \vee \psi) \rightarrow \bigvee \{\varphi, \psi\}$. In the other direction, by theorems of \L ukasiewicz logic, $\vdash \varphi \rightarrow (\varphi \vee \psi)$ and $\vdash \psi \rightarrow (\varphi \vee \psi)$. Hence, by (SUP), $\vdash \bigvee \{\varphi, \psi\}  \rightarrow (\varphi \vee \psi)$ and the claim is settled. 

\vspace*{0.2cm}
\ref{pro:basic:item3} It is easily derived from (SUP), \ref{L:item4} and the definition of $\bigwedge$.

\vspace*{0.2cm}
\ref{pro:basic:item4} It follows from \ref{S:item1}, \ref{L:item4} and the definition of $\bigwedge$.

\vspace*{0.2cm}
\ref{pro:basic:item5} For any $k \in\mathbb{N}$, $\vdash \psi_k \to \bigvee_{n\in \mathbb{N}}\psi_n$, hence by the congruence property of $\odot$, $\vdash (\varphi \odot \psi_k) \to \left(\varphi \odot \bigvee_{n\in \mathbb{N}}\psi_n\right)$. By (SUP), $\vdash \bigvee_{n\in \mathbb{N}}\left( \varphi\odot \psi_n\right)\to \left(\varphi  \odot \bigvee_{n\in \mathbb{N}} \psi_n\right)$.

To prove the converse implication, we have the following:

$\vdash \neg \bigvee_{n\in \mathbb{N}}(\varphi \odot \psi_n) \leftrightarrow \bigwedge_{n\in \mathbb
N} \neg (\varphi \odot \psi_n)$, by definition of $\bigwedge$;

$\vdash \bigwedge_{n\in \mathbb
N} \neg (\varphi \odot \psi_n) \leftrightarrow  \bigwedge_{n\in \mathbb
N} (\varphi \rightarrow \neg \psi_n)$, by definition of $\odot$;

$\vdash \bigwedge_{n\in \mathbb N} (\varphi \rightarrow \neg \psi_n) \leftrightarrow \bigwedge_{n\in \mathbb N} (\psi_n \rightarrow \neg \varphi)$, by \ref{L:item4};

$\vdash \bigwedge_{n\in \mathbb N} (\psi_n \rightarrow \neg \varphi) \rightarrow \left( \psi_k \rightarrow \neg \varphi\right)$, by \ref{pro:basic:item4} and for any $k \in \mathbb{N}$;

hence $\vdash \neg \bigvee_{n\in \mathbb{N}}(\varphi \odot \psi_n) \rightarrow \left( \psi_k \rightarrow \neg \varphi\right)$, for any $k \in \mathbb{N}$.

It follows that $\vdash \psi_k \rightarrow \left( \neg \bigvee_{n\in \mathbb{N}}(\varphi \odot \psi_n) \rightarrow \neg \varphi \right) $ for any $k \in \mathbb{N}$ by the exchange law. By (SUP) $\vdash \bigvee_{n\in \mathbb{N}}\psi_n \rightarrow \left( \neg \bigvee_{n\in \mathbb{N}}(\varphi \odot \psi_n) \rightarrow \neg \varphi \right) $,

and again by the exchange law, $\vdash \neg \bigvee_{n\in \mathbb{N}}(\varphi \odot \psi_n) \rightarrow \left( \bigvee_{n\in \mathbb{N}}\psi_n \rightarrow \neg \varphi \right)$.

Finally, since $\vdash \left( \bigvee_{n\in \mathbb{N}}\psi_n \rightarrow \neg \varphi \right) \rightarrow \left( \varphi \rightarrow \neg \bigvee_{n\in \mathbb{N}}\psi_n \right) $, and

$\vdash  \left( \varphi \rightarrow \neg \bigvee_{n\in \mathbb{N}}\psi_n \right) \leftrightarrow \neg \left(  \varphi \odot \bigvee_{n\in \mathbb{N}}\psi_n \right)$, we get

$ \vdash \neg \bigvee_{n\in \mathbb{N}}(\varphi \odot \psi_n) \rightarrow  \neg \left(  \varphi \odot \bigvee_{n\in \mathbb{N}}\psi_n \right)$.

A final application of \ref{L:item4} entails $\vdash\left(  \varphi \odot \bigvee_{n\in \mathbb{N}}\psi_n \right)\rightarrow \bigvee_{n\in \mathbb{N}}(\varphi \odot \psi_n)$, which settles the claim.

\vspace*{0.2cm}
\ref{pro:basic:item7} Let us apply \ref{pro:basic:item5} to the sequence $\{ \neg \psi_n\}_{n\in \mathbb{N}}$. We have

$\vdash \left(\varphi  \odot \bigvee_{n\in \mathbb{N}} \neg \psi_n\right)\leftrightarrow \bigvee_{n\in \mathbb{N}}\left( \varphi\odot \neg \psi_n\right)$, and consequently,

$\vdash \neg \left(\varphi  \odot \bigvee_{n\in \mathbb{N}} \neg \psi_n\right)\leftrightarrow \neg \bigvee_{n\in \mathbb{N}}\left( \varphi\odot \neg \psi_n\right)$.\\
By definition of $\bigwedge$,

$\vdash \neg \left(\varphi  \odot \neg \bigwedge_{n\in \mathbb{N}} \psi_n\right)\leftrightarrow  \bigwedge_{n\in \mathbb{N}}\neg \left( \varphi\odot \neg \psi_n\right)$, and by definition of $\varphi\odot\psi$ as $\neg (\varphi\rightarrow \neg \psi)$, 

$\vdash \left(\varphi  \rightarrow \bigwedge_{n\in \mathbb{N}} \psi_n\right)\leftrightarrow  \bigwedge_{n\in \mathbb{N}}\left( \varphi\rightarrow \psi_n\right)$.
\end{proof}

\begin{remark}\label{rem:sost-equiv}
\Cref{pro:basic}\ref{pro:basic:item1} implies that \NL \ has the property of substitution of equivalents. It follows from the fact that the logic $\mathcal{R}$\L \ has this property and the fact that in the proof of \Cref{pro:basic}\ref{pro:basic:item1} we have used the property on the level of the connectives of Riesz logic. As a consequence, we can safely use theorems that hold in the logic $\mathcal{R}$\L. \end{remark}

\begin{definition}\label{def:limit}
Given a sequence of formulas $\{ \varphi_n \}_{n\in \mathbb{N}}\subseteq FORM_{IRL}$, we say that $\varphi\in FORM_{IRL}$ is the \emph{order limit} of $\{ \varphi_n\}_{n}$ and write $\varphi=\lim^o_n \varphi_n$ if, and only if, there exists an increasing sequence of formulas $\{ \psi_n \}_{n \in \mathbb{N}}\subseteq FORM_{IRL}$ such that $\vdash \bigvee_{n\in \mathbb{N}}\psi_n$ and $\vdash \psi_n \rightarrow (\varphi \leftrightarrow \varphi_n)$ for any $n\in \mathbb{N}$.
\end{definition}
\begin{remark}
As will be explored in more detail in Section \ref{sec:KHSpace}, Definition \ref{def:limit} is reflected in the Lindenbaum-Tarski algebra in the notion of order convergence, while a stronger definition can be proved to be equivalent to the uniform convergence of term functions.
\end{remark}

\begin{proposition}\label{pro:limitProperties}
\begin{enumerate}[label=(\roman*)]
\item \label{pro:limitProperties:item1} If $\varphi=\lim^o_n \varphi_n$, then $\lim^o_n \neg \varphi_n=\neg \varphi$;
\item \label{pro:limitProperties:item2}  If $\varphi=\lim^o_n \varphi_n$ and $\psi=\lim^o_n \psi_n$, then $\lim^o_n (\varphi_n \odot \psi_n)=\varphi\odot \psi$, where $\varphi\odot \psi$ is defined as $\neg (\varphi\rightarrow \neg\psi)$;
\item \label{pro:limitProperties:item3}  Let $\varphi$ and $\psi$ be both limits of a sequence of formulas $\{ \varphi_n \}_{n}\subseteq FORM_{IRL}$. Then $\vdash \varphi \leftrightarrow \psi$;
\item \label{pro:limitProperties:item4}  If $\varphi=\lim^o_n \varphi_n$ and $\vdash \varphi\leftrightarrow \psi$, then $\psi=\lim^o_n \varphi_n$;
\item \label{pro:limitProperties:item5}  If $\varphi=\lim^o_n \varphi_n$ and $\psi=\lim^o_n \psi_n$, then $\lim^o_n (\varphi_n \oplus \psi_n)=\varphi\oplus \psi$, where $\varphi\oplus \psi$ is defined as $\neg \varphi\rightarrow \psi$.
\end{enumerate}
\end{proposition}
\begin{proof}
\ref{pro:limitProperties:item1} By hypothesis, there exists a sequence $\{ \psi_n \}_{n}\subseteq FORM_{IRL}$ such that $\vdash \bigvee_{n}\psi_n$ and $\vdash \psi_n \rightarrow (\varphi \leftrightarrow \varphi_n)$. Since $\vdash (\varphi \leftrightarrow \varphi_n)\rightarrow(\neg \varphi \leftrightarrow \neg \varphi_n)$, we deduce $\vdash \psi_n \rightarrow (\neg \varphi_n \leftrightarrow \neg  \varphi)$ and the claim is settled.

\vspace*{0.2cm}
\ref{pro:limitProperties:item2} By hypothesis, there exist increasing sequences $\{ \chi_n \}_{n},\ \{ \varrho_n \}_{n} \subseteq FORM_{IRL}$ such that $\vdash \bigvee_{n}\chi_n$, $\vdash \bigvee_{n}\varrho_n$ and $\vdash \chi_n \rightarrow (\varphi \leftrightarrow \varphi_n)$ and $\vdash \varrho_n \rightarrow (\psi \leftrightarrow \psi_n)$.\\
From the theorems of \L, 
\[\vdash (\sigma_1 \rightarrow \gamma_1)\odot (\sigma_2 \rightarrow \gamma_2)\rightarrow (\sigma_1\odot \sigma_2 \rightarrow \gamma_1 \odot \gamma_2) \quad \text{ and }\quad \vdash \sigma\rightarrow(\gamma\rightarrow(\sigma\odot\gamma)),\]
we deduce $\vdash (\varrho_n \odot \chi_n) \rightarrow [(\varphi \leftrightarrow \varphi_n)\odot (\psi \leftrightarrow \psi_n)]$.

Being $\sigma\leftrightarrow \gamma$ defined as $(\sigma\rightarrow \gamma)\odot (\gamma\rightarrow \sigma)$, by commutativity of $\odot$, we deduce 
\[\vdash [(\varphi \leftrightarrow \varphi_n)\odot (\psi \leftrightarrow \psi_n)] \leftrightarrow [(\varphi \odot \psi)\leftrightarrow (\varphi_n \odot \psi_n)],\]
hence $\vdash (\varrho_n \odot \chi_n) \rightarrow  [(\varphi \odot \psi)\leftrightarrow (\varphi_n \odot \psi_n)]$. 

It remains to prove that $\vdash \bigvee_n (\chi_n \odot \varrho_n)$. It is easily seen from the distributivity of $\bigvee$ over $\odot$ that $\vdash \bigvee_n \bigvee_m (\varrho_n \odot \chi_m)\leftrightarrow \left(\bigvee_n\chi_n\odot\bigvee_m\varrho_m\right)$, and therefore $\vdash \bigvee_n \bigvee_m (\varrho_n \odot \chi_m)$. 

Since both sequences are increasing, for any $n, m\in \mathbb{N}$ let $k=\max(n,m)$, then $\varrho_n\odot \chi_m\le \varrho_k\odot\chi_k$, $\vdash \bigvee_m\bigvee_n (\chi_m \odot \varrho_n) \rightarrow \bigvee_n (\varrho_n \odot \chi_n)$ by applications of (SUP) and \Cref{pro:basic}\ref{pro:basic:item1}, thus $\vdash \bigvee_n \varrho_n \odot \chi_n $.

\vspace*{0.2cm}
\ref{pro:limitProperties:item3} Let $\varphi=\lim^o_n \varphi_n =\psi$. Hence,  there exist increasing sequences $\{ \chi_n \}_{n},\ \{ \varrho_n \}_{n} \subseteq FORM_{IRL}$ such that $\vdash \bigvee_{n}\chi_n$, $\vdash \bigvee_{n}\varrho_n$ and $\vdash \chi_n \rightarrow (\varphi \leftrightarrow \varphi_n)$ and $\vdash \varrho_n \rightarrow (\psi \leftrightarrow \varphi_n)$.

Similarly to \ref{pro:limitProperties:item2}, we get $\vdash (\varrho_n \odot \chi_n) \rightarrow  [(\varphi \leftrightarrow \varphi_n)\odot (\psi \leftrightarrow \varphi_n)]$. Since, by definition of $\leftrightarrow$, $\vdash [(\varphi \leftrightarrow \varphi_n)\odot (\psi \leftrightarrow \varphi_n)]\rightarrow(\varphi \leftrightarrow \psi)$, for any $n$ we derive $\vdash (\varrho_n \odot \chi_n) \rightarrow  (\varphi \leftrightarrow \psi)$.

By (SUP), $\vdash \bigvee_n(\varrho_n \odot \chi_n) \rightarrow  (\varphi \leftrightarrow \psi)$. Finally, as in \ref{pro:limitProperties:item2}, $\vdash \bigvee_n \varrho_n \odot \chi_n $ and the claim is settled.

\vspace*{0.2cm}
\ref{pro:limitProperties:item4} It follows easily by the properties of logical equivalence.

\vspace*{0.2cm}
\ref{pro:limitProperties:item5} It is a straightforward application of \cref{pro:limitProperties:item1,pro:limitProperties:item2,pro:limitProperties:item3,pro:limitProperties:item4}.
\end{proof}

As usual, we can consider the Lindenbaum-Tarski algebra of the logic \NL . Let $\Theta \subset FORM_{IRL}$ be a subset of formulas of \NL . The relation defined by 
$$\varphi \equiv_{\Theta} \psi \mbox{\quad if, and only if, \quad} \Theta \vdash \varphi \rightarrow \psi \mbox{ and }\Theta \vdash \psi \rightarrow \varphi,$$
is a congruence relation. Thus we define, on the quotient $FORM_{IRL}/{\equiv_{\Theta}}$, the following operations:
\begin{enumerate}
\item $0_{\Theta}=[\bot]_{\Theta}$;
\item $[\varphi]_{\Theta} \rightarrow [\psi]_{\Theta}= [\varphi \rightarrow \psi]_{\Theta}$;
\item $[\varphi ]_{\Theta}^*=[\neg \varphi]_{\Theta}$;
\item $[\varphi]_{\Theta} \oplus [\psi]_{\Theta}=[\neg \varphi\rightarrow \psi]_{\Theta}$;
\item $\alpha [\varphi]_{\Theta}=[\Delta_{\alpha}(\varphi)]_{\Theta}$, where $\Delta_{\alpha} \varphi$ stands for $\neg \nabla_{\alpha}(\neg \varphi)$;
\item $\bigvee_{n\in \mathbb{N}} [\varphi_n]_{\Theta}=[\bigvee_{n\in \mathbb{N}}\varphi_n]_{\Theta}$.
\end{enumerate}
By \cite{LeuRMV}, all operations from (1) to (5) are well defined and it is easily seen that $1_{\Theta}=0_{\Theta}^*$ coincide with the set of syntactical consequences of $\Theta$. (6) is well defined by Proposition \ref{pro:basic}(i): indeed, $\varphi_n \equiv \psi_n$ for any $n\in \mathbb{N}$ implies $\bigvee_n \varphi_n \equiv \bigvee_n \psi_n$. When $\Theta = \emptyset$, the quotient will be denoted by $IRL$. We recall that the order on $IRL$ is defined by $[\varphi]\le [\psi]$ if, and only if, $\vdash \varphi \rightarrow \psi$.

\begin{proposition}
$IRL$ is a $\sigma$-complete Riesz MV-algebra.
\end{proposition}
\begin{proof}
The fact that $IRL$ is a Riesz MV-algebra follows from the definition of the operation in (1)-(5) and \cite[Proposition 5]{LeuRMV}. It is $\sigma$-complete since by (6) we are closing the algebra under all possible countable suprema. We only need to prove that $\bigvee$ is indeed a suprema in $IRL$. By (S1), it follows that $ [\varphi_k]\le [\bigvee_{n\in \mathbb{N}}\varphi_n]=\bigvee_{n\in \mathbb{N}}[\varphi_n]$, hence $\bigvee_{n\in \mathbb{N}}[\varphi_n]$ is an upper bound. By (SUP), if for some $[\psi]$, $[\varphi_k]\le [\psi]$ for any $k$, then $\bigvee_{n\in \mathbb{N}}[\varphi_n]\le[\psi]$ and $\bigvee_{n\in \mathbb{N}}[\varphi_n]$ is indeed the supremum of $\{ \varphi_n\}_{n\in \mathbb{N}}$.
\end{proof}

Thus, a suitable semantics for this logic has  $\sigma$-complete Riesz MV-algebras as models. In the next subsection we will go deeper in the study of our models, while we close this section with a completeness result \--- that will be improved subsequently \--- and a characterization of the algebra $IRL$ in terms of $RL$, the Lindenbaum-Tarski algebra of $\mathcal{R}\L$.

We define the notion of evaluation as usual. Let $A=(A, \oplus, ^*, \{ \alpha \}_{\alpha\in [0,1]}, 0)$ be a  $\sigma$-complete Riesz MV-algebra. A map $e: FORM_{IRL}\to A$ is an evaluation if:
\begin{enumerate}[label=(\roman*)]
\item $e(\neg \varphi)= e(\varphi)^*$,
\item $e(\varphi\rightarrow \psi)=e(\varphi)^*\oplus e(\psi)$,
\item $e(\Delta_\alpha \varphi)=\alpha e(\varphi)$, where $\Delta_{\alpha}$ is the dual connective of $\nabla_{\alpha}$,
\item $e(\bigvee_{n\in \mathbb{N}} \varphi_n)=\bigvee_{n\in \mathbb{N}} e(\varphi_n)$.
\end{enumerate}

\begin{proposition}[Completeness] \label{pro:completeness}
Let $\mathbf{RMV_{\sigma DC}}$ be the category of  $\sigma$-complete Riesz MV-algebras and $\sigma$-complete morphisms and $\varphi$ a formula of \NL. Then $\vdash_{IRL} \varphi$ if and only if $e(\varphi)=1$ for any evaluation $e:FORM_{IRL}\to R$ and any $R\in \mathbf{RMV_{\sigma DC}}$.
\end{proposition}
\begin{proof}
For the left-to-right direction, it easy to check that all axioms are $R$-tautologies, for any $R\in \mathbf{RMV_{\sigma DC}}$, and that both  Modus Ponens and (SUP) lead tautologies to tautologies.

For the converse direction, it is enough to consider the evaluation $e: FORM_{IRL}\to IRL$ defined by $e(\varphi)=[\varphi]$. Being $[\varphi]=1$ if, and only if, $\vdash \varphi$, the claim is easily settled.
\end{proof}

As we shall prove in Section \ref{sec:loomis}, the system \NL\ also enjoys standard completeness. The standard model is $[0,1]$ with the natural structure of  $\sigma$-complete Riesz MV-algebra, see \Cref{inf-hsp} and \Cref{stdco}.

\begin{proposition} \label{pro:consext}
\NL \ is a conservative extension of $\mathcal{R}\L$.
\end{proposition}
\begin{proof}
The fact that \NL\ is an extension of $\mathcal{R}\L$ is a consequence of \Cref{rem:sost-equiv}. To prove that it is a conservative extension, let $\varphi$ be a formula of $\mathcal{R}\L$  which is a theorem in \NL . By completeness, $\varphi $ is a tautology in the standard RMV-algebra $[0,1]$ which is a complete Riesz MV-algebra as well. Since $\varphi$ does not involves $\bigvee$, it is a tautology in $\mathcal{R}\L$, and by completeness of $\mathcal{R}\L$, $\varphi$ is a theorem in $\mathcal{R}\L$.
\end{proof}

\begin{remark}\label{rem:infOp}
The existence of free objects in complete boolean algebras is a known problem: when one deals with infinitary disjunctions and conjunctions, they may not always exist. One well behaved case is the case in which the algebra is, in addition, completely distributive. In the case of MV-algebras, it is mentioned in \cite[Chapter 11]{MunBook} that $\sigma$-complete MV-algebras satisfy the \emph{countable} distributivity law. Another well behaved case is the one of infinitary varieties. In such a case, each operation has infinite (but bounded) arity in contrast with the case of complete algebras, where the arity of the disjunctions and conjuctions depends on the cardinality of the algebra. \emph{Infinitary universal algebra} is developed in \cite{Slomiski} and it allows us to safely discuss free objects.
\end{remark}

\begin{proposition}\label{pro:IRLfree}
$IRL$ is the free algebra in $\mathbf{RMV_{\sigma DC}}$ generated over the set of equivalence classes of propositional variables. Analogously, $IRL_n$ is freely generated by $n$ propositional variables.
\end{proposition}
\begin{proof}
The result is a straightforward consequence of \cite[Chapter III, Theorem 8.3]{Slomiski}. With the aim of showing this standard technique and for the sake of completeness, we now give a detailed proof.

Let $X$ be the set of equivalence classes for the propositional variables of \NL , i.e. $X=\{ [v]\mid v\text{ is a propositional variable}\}$. Let $R$ be a  $\sigma$-complete Riesz MV-algebra and let $f: X \rightarrow R$ be a function. We need to  prove that there exists a unique $\sigma$-homomorphism of Riesz MV-algebras $f^{\sharp}:IRL\rightarrow R$ such that $f^{\sharp} ([v])=f([v])$ for any $v\in Var$.\\
With standard arguments \--- evaluations are uniquely determined by the values on variables, see \cite[Chapter 4.3]{Karp} for the Boolean case \--- one can show that there exists a unique evaluation $e$ such that $e(v)=f([v])$ for any $v\in Var$. If $\varphi $ and $\psi$ are formulas such that $\varphi \equiv \psi$, we have $e(\varphi)=e(\psi)$ and $\equiv \subset Ker(e)$. We recall that $\equiv$ is a congruence relation on $FORM_{IRL}$, and by the homomorphisms theorem \cite[Chapter II, Theorem 5.8]{Slomiski} there exists $f^{\sharp}: \quot{FORM_{IRL}}{\equiv} \rightarrow R$ such that $f^{\sharp}([\varphi])=e(\varphi)$ for any $\varphi \in FORM_{IRL}$. $f^{\sharp}$ is a $\sigma$-homomorphism  of Riesz MV-algebras, indeed:
\begin{enumerate}[label=(\roman*)]
\item $f^{\sharp}([\varphi]^*)=f^{\sharp} ([\neg \varphi])=e(\neg \varphi)=e(\varphi)^*=(f^{\sharp}([\varphi]))^*$;
\item $f^{\sharp}([\varphi]\oplus [\psi])= f^{\sharp}([\neg \varphi \rightarrow \psi])= e(\neg \varphi \rightarrow \psi)= e(\varphi)\oplus e(\psi)= f^{\sharp}([\varphi])\oplus f^{\sharp}([\psi])$;
\item $f^{\sharp}(\mathbf{1})=f^{\sharp}([\varrho])=e(\varrho)=1$, where $\varrho$ is a theorem of \NL ;
\item $ f^{\sharp}(\alpha[\varphi])=f^{\sharp}([\Delta_\alpha \varphi])= e(\Delta_\alpha \varphi)=\alpha e(\varphi)=\alpha f^{\sharp}([\varphi])$;
\item $f^{\sharp}(\bigvee_n[\varphi_n])=f^{\sharp} ([\bigvee_n \varphi_n])=e(\bigvee_n \varphi_n)=\bigvee_ne(\varphi_n)^*=\bigvee_nf^{\sharp}([\varphi_n])$. 
\end{enumerate}
It remains to prove the uniqueness of $f^{\sharp}$. Let $g: IRL\rightarrow R$ be another $\sigma$-homomorphism of Riesz MV-algebras such that $g([v])=f([v])$ for any $v\in Var$. Let $\pi$ be the canonical surjection $\pi: FORM_{IRL}\rightarrow FORM_{IRL}/\equiv$; it is obvious that $g \circ \pi $ is an evaluation such that $(g\circ \pi )(v)=e(v)$, then $g\circ \pi =e$ (as $e$ was the unique evaluation with this property) and $g([\varphi])=e(\varphi)$ for any $\varphi \in FORM_{IRL}$; since $f^{\sharp}$ was the unique extension with this property, $g=f^{\sharp}$.
\end{proof}

\begin{proposition}\label{pro:IRLisDsC}
$IRL_n$ is the  $\sigma$-completion of $RL_n$, where $RL_n$ denotes the Lindenbaum-Tarski algebra of $\mathcal{R}\L$ on $n$ propositional variables.
\end{proposition}
\begin{proof}
By Proposition \ref{pro:consext}, $\iota: RL_n\hookrightarrow IRL_n$ and the embedding is defined by $[v_i]_{RL}\mapsto [v_i]_{IRL}$ for any propositional variable $v_i$, with $i=1, \dots ,n$.

We need to prove that, for any $f:RL_n\to R$, where $R$ is a  $\sigma$-complete Riesz MV-algebra, there exists a unique $g:IRL_n\to R$ such that $g\mid_{RL_n}=f$. 

From $f$ we can define a map $\overline{f}:\{[v_1]\dots , [v_n] \}\to R$ by $\overline{f}([v_i])=f([v_i]_{RL})$. By Proposition \ref{pro:IRLfree} there exists a unique $\sigma$-homomorphism of Riesz MV-algebras $f^{\sharp}:IRL_n\to R$ such that $f^{\sharp}([v_i])=f^{\sharp}(\iota([v_i]_{RL}))=f([v_i]_{RL})$. Being $RL_n$ freely generated by $[v_i]_{RL}$, $f^{\sharp}$ agrees with $f$ over $\iota(RL_n)$ and the claim is settled.
\end{proof}

Thus, $IRL_n$ is the  $\sigma$-completion of $RL_n$ \--- which is the algebra of piecewise linear functions with real coefficients \--- and $RL_n$ has $C([0,1]^n)$ \--- which is the algebra of $[0,1]$-valued continuous functions defined over $[0,1]^n$ \--- as its norm completion, see \cite[Theorem 2.15]{DiNLL-RMV}. Since norm completions, Dedekind $\sigma$-completions and Dekekind completions of Riesz Spaces are a deeply investigated subject, we shall see in the next subsection how this can be exploited in a logical setting. 

\begin{remark}
To obtain \NL \ we have started from the logic $\mathcal{R}\L$ of Riesz MV-algebras. To the same effect, we could have started from the \emph{Rational \L ukasiewicz logic} $\mathcal{Q}\L$ defined by B. Gerla in \cite{brunella}. This logical system stands in between \L ukasiewicz logic and $\mathcal{R}\L$: in \cite{LLDMV} the authors provided an axiomatization for $\mathcal{Q}\L$ that has the same axioms in \cref{R:item1,R:item2,R:item3,R:item4}, but scalars are taken in $[0,1]\cap \mathbb{Q}$ instead of $[0,1]$. Thus, starting with $\{ \nabla_q\mid q\in [0,1]\cap \mathbb{Q}\}$ and $\bigvee$, we obtain the same logical system by the following remark:
\[ \Delta_r\varphi =\bigvee\{ \Delta_q\varphi \mid r>q\in [0,1]\cap\mathbb{Q}\}, \]
where $\Delta_r\varphi=\neg \nabla_r \neg \varphi$.
\end{remark}

\section{Limits, compact Hausdorff spaces and completions}\label{sec:KHSpace}
In this section we investigate the models of \NL \ from a different point of view. Indeed, in Definition \ref{def:limit} we have defined a notion of \emph{logical order convergence} and in Proposition \ref{pro:limitProperties} we have proved that, in the Lindenbaum-Tarski algebra of \NL , the limit is unique and the MV-algebraic operations are continuous with respect to the limit. It was proved in \cite{DiNLL-RMV} that in $RL_n$, the Lindenbaum-Tarski algebra of the logic $\mathcal{R}\L$, Definition \ref{def:limit} can be strengthened in such way that it corresponds to the notion of \emph{uniform convergence}: if we denote by $f_{\psi}$ the function that corresponds (via isomorphism) to $[\psi]$ in $RL$, we have the following theorem.
\begin{theorem} \cite[Theorem 2.8]{DiNLL-RMV}\label{teo:limitREF}
 Let $\{\varphi_n\}_{n\in\mathbb{N}}\subseteq Form_{\mathcal{R}L}$ and $\varphi\in Form_{\mathcal{R}L}$. Then the following are equivalent:\\
(i)  $\{f_{\varphi_n}\}_{n\in\mathbb{N}}$ converges uniformly to $f_\varphi$,\\
(ii)  There exists an increasing sequence of real numbers $r_n$ such that $\bigvee r_n=1$ and $\vdash \eta_{r_n}\rightarrow (\varphi_n \leftrightarrow \varphi)$ for any $n\in \mathbb{N}$, where $\eta_{r_n}=\Delta_{r_n}\top$.
\end{theorem}
Such a definition is very similar to \Cref{def:limit}, but it is stronger, as it requires a peculiar sequence of formulas, the point being that formulas of the shape $\eta_r$ correspond to constant functions in $RL$. The difference arises from the remark that in a space of functions, the pointwise infimum (or supremum) of a sequence of formulas does not need to coincide with the infimum (or supremum) of the sequence, while this distinction is immaterial in the case of constant functions. More details on this subtle difference can be found in \cite[Remark 2.9]{DiNLL-RMV}.

Thus, it is natural to ask how having an infinitary disjunction affects uniform convergence. The following proposition shows that the infinitary disjunction forces our models to be norm-complete. To do so, let us first recall that on semisimple Riesz MV-algebras the unit seminorm is actually a norm and that the full subcategory of norm-complete Riesz MV-algebras is equivalent to the one of \emph{$M$-spaces}, where an $M$-space is, via Kakutani's duality \cite{kakutani}, a norm-complete Riesz space that is isomorphic with $C(X)=\{f\colon X\to \mathbb{R}\mid f \text{ continuous}\}$, for a  compact Hausdorff space $X$.
\begin{proposition}\label{pro:dcAREnc}
Let $R$ be a Riesz MV-algebra. If $R$ is  $\sigma$-complete, then it is norm-complete with respect to the norm induced by the unit, that is $\lVert x \rVert_u=\min \{ \alpha \in [0,1]\mid x\le \alpha 1_R \}$.  As a consequence, $IRL$ is a norm-complete Riesz MV-algebra.
\end{proposition}
\begin{proof}
We first remark that any Dedekind $\sigma$-complete Riesz Space with a strong unit is an M-space (see \cite[Theorem 45.4]{RS} and subsequent discussion). It is known that an MV-algebra is $\sigma$-complete if, and only if, the corresponding $\ell u$-group is Dedekind $\sigma$-complete, as proved in \cite[Proposition 16.9]{goodearl}. Since the lattice reduct of a Riesz MV-algebra is the same of its MV-algebra reduct, we can say that a Riesz MV-algebra $R$ is $\sigma$-complete if, and only if,  the corresponding Riesz Space $(V_R, u)$ is Dedekind  $\sigma$-complete. By \cite[Theorem 7]{LeuRMV}, $(V_R,u)$ is an M-space if and only if $R$ is norm-complete, and the claim is settled.
\end{proof}

In \cite{DiNLL-RMV} we have discussed the norm-completion of $RL_n$ with respect to the unit norm above mentioned. The previous proposition linked closely together norm-completion and  $\sigma$-completion of $RL_n$, indeed we have the following.

\begin{proposition}\label{contsub}
For any $n\in \mathbb{N}$, denoted by $IRL_n$ the Lindenbaum-Tarski algebra built upon formulas with at most $n$ variables and up to isomorphism, $C([0,1]^n)\leq IRL_n$.
\end{proposition}
\begin{proof}
By Proposition \ref{pro:dcAREnc}, $IRL_n$ is a norm-complete Riesz MV-algebra. The unit norm of continuous functions coincides with the sup-norm and by \cite[Theorem 3.2]{DiNLL-RMV} the norm-completion of $(RL_n, \lVert \cdot \rVert_u)$ is $(C([0,1]^n), \lVert \cdot \rVert_\infty)$. Thus, since $RL_n\subseteq IRL_n$, $IRL_n$ must contain the norm-completion of $RL_n$.
\end{proof}

\begin{remark} The converse of Proposition \ref{pro:dcAREnc} does not hold. Indeed there are M-spaces that are not Dedekind $\sigma$-complete. An  example is $C([0,1])$, see \cite[Example 23.3(ii)]{RS}. Thus, one can see that $C([0,1]^n)$ is not Dedekind $\sigma$-complete.
\end{remark}

Finally, we recall the following theorem.
\begin{theorem}\cite[Corollary 5]{LeuRMV}\label{teo:dualRMVKHaus}
The categories $\mathbf{URMV}$ of norm complete Riesz MV-algebras (full subcategory of $\mathbf{RMV}$) and $\mathbf{KHaus}$ of compact Hausdorff spaces and continuous maps are dually equivalent. 
\end{theorem}

Whence, building on Theorem \ref{teo:dualRMVKHaus} and Proposition \ref{pro:dcAREnc} we can say that, for any model $R$ of \NL, there exists a compact Hausdorff space $K$ such that $R\simeq C(K)$, the unit interval of the algebra of continuous functions defined over $K$. Since the models of the logic \NL \ have the extra property of being $\sigma$-complete, we can say something more, as proved in the following theorem.\\

Let $X$ be a topological space. We recall that a subset of $X$ is said to be $F_{\sigma}$ if it is the union of countably many closed subsets of $X$. Moreover, $X$ is called \emph{basically disconnected} if the closure of every open $F_{\sigma}$ subset is open. In some books and papers, basically disconnected spaces are also called \emph{quasi-Stonean}. We also recall that a compact Hausdorff space is basically disconnected if, and only if, it is a totally disconnected space in which the union of any countable family of clopens is clopen. Whence, any basically disconnected compact Hausdorff space is a Stone space. 

\begin{theorem} \label{teo:quasiStonean}
A norm-complete Riesz MV-algebra is $\sigma$-complete if, and only if, the corresponding compact Hausdorff space is basically disconnected.
\end{theorem}
\begin{proof}
It is straightforward by \cite[Proposition 2.1.5]{banachlattices} and \cite[Corollary 5]{LeuRMV}.
\end{proof}

We conclude this section with three remarks. \Cref{pro:dcAREnc} and \Cref{teo:quasiStonean} entail an important property: it is possible to obtain a functional representation for the algebra $IRL_n$ by continuous functions. Indeed, by the previous argument and Kakutani's duality, it must exist a basically disconnected space $K$ such that the $\sigma$-complete algebra $IRL_n$ is isomorphic with $C(K)$. Although this is an important characterization, we aim at a characterization of $IRL_n$ by (non-necessarily continuous) functions defined on $[0,1]^n$. This is the goal of the final part of this paper. Moreover, as a corollary of \Cref{teo:quasiStonean} and \Cref{pro:completeness} we can say that the logic \NL \ is complete with respect to models of the type $C(X)$, where $X$ is a basically disconnected compact Hausdorff space.

Finally, we recall a different approach to the idea of finding a logic whose models are closely related to compact Hausdorff spaces. To do so, we set the following: for any sequence $\varphi_1, \varphi_2,\ldots$ we define the sequence $\sigma_1=\Delta_{\frac{1}{2}}\varphi_1$, $\sigma_2=\Delta_{\frac{1}{2}}\varphi_1\oplus \Delta_{\frac{1}{2^2}}\varphi_2,\ldots$ and 
we set $\mathbf{\delta}(\varphi_1,\varphi_2,\cdots)=lim_n \sigma_n$.  

The operator denoted by $\delta$ has been introduced in a much more general way in \cite{EnzoReggio} and MV-algebras endowed with a $\delta$-operator are in categorical duality with compact Hausdorff spaces. In particular, the main result from  \cite{EnzoReggio} proves that the following categories are equivalent:\\
(CK) the dual of the category of compact Hausdorff spaces with continuous
maps;\\
(CD) the category of $\delta$-algebras with $\delta$-preserving MV-algebra morphisms. \\
Thus, the category of $\delta$-algebras is equivalent to the category  of norm-complete Riesz MV-algebras with  homomorphisms of Riesz MV-algebras.

\begin{remark}[From $\sigma$-complete to complete algebras]\label{rem:dedcomplete}
A natural question to ask is what happens if we allow arbitrary disjunctions: in this case, our class of interest will be the one of complete Riesz MV-algebras. Following \cite[Theorem 9.28]{Alip} and  \cite[Lemma 1.1]{kakutani} we get that the algebra is again (isomorphic with) the unit interval of an M-space. By \cite[Proposition 2.1.4]{banachlattices}, the compact space associated to such algebra is an extremely disconnected compact Hausdorff space. Finally,  we can push the logical setting a little further defining a disjunction of arity $\kappa$, where $\kappa$ is a fixed cardinal. The main complication in the case of complete Riesz MV-algebras is that we will lose the power of the approach of \cite{Slomiski} for infinitary varieties: in this case we would deal with disjunctions that depend on the cardinality of the algebra, and therefore they do not constitute a proper class of infinitary operations in the sense of \cite{Slomiski}.
\end{remark}

\section{The Loomis-Sikorski theorem, standard completeness and concrete representation}\label{sec:loomis}
In this section we prove the analogue of the Loomis-Sikorski theorem   for Riesz MV-algebras, which allows us to prove the standard completeness of \NL . We also represent the Lindenbaum-Tarski algebra in $n$ variables as an algebra of $[0,1]$-valued functions defined over $[0,1]^n$.

\begin{definition}
A \emph{Riesz tribe} over $X$ is a Riesz MV-algebra of $[0,1]$-valued functions over $X$ that is closed under \emph{pointwise} countable suprema. We remark that 
any Riesz tribe is a  $\sigma$-complete Riesz MV-algebra.
\end{definition}
Let $\mathcal{T}$ be a Riesz tribe and let $A$ be a  $\sigma$-complete Riesz MV-algebra. A map $\eta\colon \mathcal{T}\to A$ is a $\sigma$-homomorphism if, and only if, for each sequence $\{f_1, f_2, \dots, f_k, \dots \}\subseteq \mathcal{T}$, if $\sup_k f_k$ is the pointwise suprema of the sequence, then $\eta(\sup_k f_k)=\bigvee_k \eta(f_k)$.

Finally, let $X$ be a topological space. A subset $Z$ of $X$ is said to be \emph{meager} if it is the union of countably many subsets of $X$ whose closure has empty interior. Any two functions $f,g\in[0,1]^X$ are said to be \emph{essentially equal}, in symbols $f\approx g$, if $\{ x\in X \mid f(x)\neq g(x)\}$ is meager.

We recall that such notions have been defined in the theory of MV-algebras in \cite{Dvu-LS, Mun} while a more recent account can be found in \cite[Chapter 11]{MunBook}. Moreover, we recall that any  $\sigma$-complete Riesz MV-algebra $A$ is semisimple and we shall identify it with a subalgebra of $C(Max(A))$, where $Max(A)$ is the space of maximal ideals of $A$, see for example \cite{CDM} for the case of MV-algebras and recall that a Riesz MV-algebra is semisimple if, and only if, its MV-algebra reduct is semisimple.

\begin{theorem}[The Loomis-Sikorski Theorem for Riesz MV-algebras]\label{teo:LS}
Let $A\subseteq C(X)$ be a $\sigma$-complete Riesz MV-algebra, where $X=Max(A)$, and let $\mathcal{T}\subseteq [0,1]^X$ be the set of functions $f$ that are essentially equal to some function of $A$. Then $\mathcal{T}$ is a Riesz tribe,  each $f\in \mathcal{T}$ is essentially equal to a unique $f^*\in A$ and the map $f\mapsto f^*$ is a $\sigma$-homomorphism of $\mathcal{T}$ onto $A$.
\end{theorem}
\begin{proof}
The proof is easily deduced from the MV-algebraic version of the Loomis-Sikorski theorem, \cite[Theorem 11.7]{MunBook}, applied to the MV-algebra reduct of $A$. 
We only need to check that $\mathcal{T}$ is closed under the scalar operation and that the map $\eta: f\mapsto f^*$ is an homomorphism of Riesz MV-algebras. The latter claim follows from the fact that MV-homomorphisms between semisimple Riesz MV-algebras are Riesz MV-homomorphisms, while the former is easily deduced by computation. Indeed, if $f\in \mathcal{T}$ there exists $f^* \in A$ such that $f\approx f^*$, which means $\{ x\in X \mid f(x)\neq f^*(x)\}$ is meager. By the properties of the scalar products, for any $\alpha>0$, $f(x)\neq f^*(x)$ if, and only if, $\alpha f(x)\neq \alpha f^*(x)$, thus $\{ x\in X \mid \alpha f(x)\neq \alpha f^*(x)\}=\{ x\in X \mid f(x)\neq f^*(x)\}$ and therefore is meager. Being $A$ a Riesz MV-algebra, $\alpha f^*\in A$ and $\alpha f\in \mathcal{T}$.
\end{proof}

In the following we make use of the infinitary version of Birkhoff's theorem \cite[Chapter III, Theorem 9.4]{Slomiski} and we prove that the class of  $\sigma$-complete Riesz MV-algebras is an infinitary  variety. 

\begin{theorem}\label{inf-hsp}
 $\mathbf{RMV_{\sigma DC}}=HSP([0,1])$, where $[0,1]$ has the natural structure of $\sigma$-complete Riesz MV-algebra. 
\end{theorem}
\begin{proof} By the Loomis-Sikorski theorem, any algebra from $\mathbf{RMV_{\sigma DC}}$ is the homomorphic image of a Riesz tribe and, by definition, any Riesz tribe belongs to $SP([0,1])$.
\end{proof}

The above  theorem actually  states the standard completeness of \NL. 

\begin{corollary}(Standard completeness)\label{stdco}
If $\varphi$ is a  formula of \NL , then  $\vdash_{IRL} \varphi$ if and only if $e(\varphi)=1$ for any evaluation $e:FORM_{IRL}\to [0,1]$.
\end{corollary}

In the final part of this section, we give a concrete description of the Lindenbaum-Tarski algebra $IRL_n$ as an algebra of $[0,1]$-valued functions. Recall that we  proved, up to isomorphism, the following inclusions
\begin{equation}\label{eq:incl}
RL_n \leq C([0,1]^n)\leq IRL_n
\end{equation}
Note that the second inclusion was proved in \Cref{contsub}.\\

We start by characterizing $IRL_n$ in terms of Riesz tribes.

It follows from \cite[Chapter III, \S\ 2]{Slomiski} that $Term_{RMV\sigma}$, the set of terms in the language of $\mathbf{RMV_{\sigma DC}}$, is the absolutely free algebra in the same language. Let us denote by $Term_{RMV\sigma}(n)$ the absolutely free algebra generated by a set of $n$ propositional variables, namely $\{v_1, \dots v_n\}$. If fix an algebra $A\in \mathbf{RMV_{\sigma DC}}$, to each $\tau \in Term_{RMV\sigma}(n)$ it corresponds a function $f_{\tau}^A:A^n\to A$, inductively defined starting from the assignment $v_i\mapsto \pi_i^A$, where $\pi_i^A:A^n\to A$ is the $i^{th}$ projection. Let us denote the algebra of term functions evaluated into $A$ by $\mathcal{RT}_n^{A}$. Since all operations are pointwise defined upon the operations of $A$, we deduce that $\mathcal{RT}_n^{A}\in \mathbf{RMV_{\sigma DC}}$. When $A=[0,1]$, we denote $\mathcal{RT}_n^{A}$  by $\mathcal{RT}_n$, $f_{\tau}^A$ by $f_{\tau}$ and $\pi_i^A$ by $\pi_i$ for any $i=1,\dots, n$. 

\begin{proposition}\label{pro:term-funct}
For any $n\in \mathbb{N}$, $\mathcal{RT}_n$ is the smallest Riesz tribe that contains the projection functions $\{ \pi_1, \dots, \pi_n\}$.
\end{proposition}
\begin{proof}
The algebra $\mathcal{RT}_n$ is a Riesz tribe: it is a subalgebra of $[0,1]^{[0,1]^n}$ and the operations defined pointwise. The fact that $\{ \pi_1, \dots, \pi_n\}$ is a set of generators for $\mathcal{RT}_n$ is straightforward by the inductive definition of a term-function.
\end{proof}

\begin{theorem}\label{teo:IRL=RT}
The algebra $\mathcal{RT}_n$ is isomorphic to the Lindenbaum-Tarski algebra $IRL_n$ of the logic \NL .
\end{theorem}
\begin{proof}
By \cite[Chapter II, Theorem 7.1]{Slomiski} it is enough to prove that $\mathcal{RT}_n$ is freely generated in $\mathbf{RMV_{\sigma DC}}$ by $\{\pi_1, \dots, \pi_n\}$.

Let $A$ be any $\sigma$-complete Riesz MV-algebra and let $\eta: \{\pi_1, \dots, \pi_n\}\to A$ be any function, where $\pi_i:[0,1]^n\to [0,1]$. For any $\tau \in Term_{RMV\sigma}(n)$, we can define $f_{\tau}^A:A^n\to A$, an element in $\mathcal{RT}_n^{A}$. Let $a_i=\eta(\pi_i)$, for any $i=1, \dots ,n$ and let $F_\eta:Term_{RMV\sigma}(n)\to A$ the map defined by $F_\eta(\tau)=f_{\tau}^A(a_1, \dots ,a_n)$. By \Cref{inf-hsp}, if two terms are equals when evaluated in $[0,1]$, then they are also equal when evaluated in $A$. Thus, since $F_\eta(v_i)=\pi_i^A(a_1, \dots ,a_n)=a_i=\eta(\pi_i)$, we deduce that $F_\eta$ induces a well defined extension $\overline{\eta}$ of $\eta$ defined by $\overline{\eta}(f_{\tau})=F_\eta(\tau)$. It is easily seen that $\overline{\eta}$ is an homomorphism of $\sigma$-complete Riesz MV-algebras. Finally, the uniqueness of $\overline{\eta}$ follows from \Cref{pro:term-funct}.
\end{proof}

Let $RL_n=\{f:[0,1]^n\to [0,1]\mid \, f \mbox{ continuous piecewise linear function}\}$, as defined in \Cref{sec:prelimRMV}. By \cite[Theorem 10]{LeuRMV},  the piecewise continuous functions are exactly the functions corresponding to the formulas of $\mathcal{R}$\L, so $RL_n\subseteq \mathcal{RT}_n$.

\begin{corollary}\label{corgencont}
$\mathcal{RT}_n$ is generated by $C([0,1]^n)$.
\end{corollary}
\begin{proof}
We have $RL_n\subseteq \mathcal{RT}_n$ and, by Propositions \ref{pro:dcAREnc}, $\mathcal{RT}_n$ is norm-complete with respect to the unit-norm, so   $C([0,1]^n)\subseteq \mathcal{RT}_n$. If $\mathcal T$ is another Riesz tribe  such that $C([0,1]^n)\subseteq {\mathcal T}$, then $\{\pi_1,\ldots, \pi_n\}\subseteq {\mathcal T}$, so $\mathcal{RT}_n\subseteq {\mathcal T}$. 
\end{proof}

To give a more concrete representation of $\mathcal{RT}_n$, let us start with two important definitions. 
\begin{definition}
Assume that  $X$ and $Y$  are metric spaces.
\begin{enumerate}[label=(\roman*)]
\item \label{def:Borel} Let Borel$(X,Y)$ be the set of the functions $f:X\to Y$ that are Borel-measureable. That is, for any open set $O$ of $Y$, $f^{-1}(O)$ is a Borel set of $X$, where  the $\sigma$- algebra of Borel sets of $X$ is the smallest $\sigma$-algebra containing the open sets. In what follows we will denote $\text{Borel}(n)=\text{Borel}([0,1]^n, [0,1])$ and we assume that it is endowed with the obvious structure of Riesz MV-algebra.
\item \label{def:Baire} Let  Baire$(X,Y)$ be the set of all Baire functions $f:X\to Y$, i.e. the smallest set that contains the continuous functions and it is closed under pointwise limits of convergent sequences of  functions belonging to it. We set $\text{Baire}(n)=\text{Baire}([0,1]^n, [0,1])$ and we assume that it is endowed with the obvious structure of Riesz MV-algebra.
\end{enumerate}
\end{definition}

\begin{proposition} \label{rem:Bo=Ba}  $IRL_n\simeq \mathcal{RT}_n\simeq \emph{Borel}(n)\simeq \emph{Baire}(n)$.
\end{proposition}
\begin{proof}
The isomorphism $\text{Borel}(n)\simeq \text{Baire}(n)$ is a consequence of  the Lebesgue-Hausdorff theorem \cite[p.393]{Kuratowski}. Since tribes are semisimple MV-algebras, an isomorphism between tribes is actually a Riesz isomorphism, so  $\mathcal{RT}_n\simeq \text{Baire}(n)$ follows by Corollary \ref {corgencont} and \cite[Proposition 3.3]{Dvu-LS}.  Finally, the isomorophism $IRL_n\simeq \mathcal{RT}_n$ is \Cref{teo:IRL=RT}.
\end{proof}

In the final part of this section we show how the isomorphism between $\mathcal{RT}_n$ and $\text{Borel}(n)$ (and therefore \Cref{rem:Bo=Ba}) can be proved in a direct manner. We believe that a direct proof can be useful to understand exactly how the Borel-measurable functions are obtained starting from formulas in $\mathcal{R}\L$. At the same time we provide a concrete example of a sequence of functions in $RL_n$ whose pointwise supremum is not even a continuous function, namely the sequence  $\{ f_{m,r}\}_{m\in \mathbb{N}}$ used in the proof of \Cref{teo:tribe=Bo}. To do so, let us denote by $\chi_E \colon [0,1]^n\to [0,1]$ the characteristic function of the subset $E\subseteq [0,1]^n$. A \emph{simple} function is a linear combination, with real scalars, of characteristic functions.

It is well known that Borel functions can be written as limits of simple functions built upon the characteristic functions of Borel subsets of $[0,1]^n$. More precisely, one can prove the following.

\begin{lemma} \label{lem:Borel=supSimple}
Any Borel function $f:[0,1]^n\to [0,1]$ is the uniform limit of an increasing sequence of simple functions $f_m:[0,1]^n\to [0,1]$, where $f_m=\sum_{i=1}^{k_m} \alpha_i\chi_{E_i}$ with $\alpha_i\in [0,1]$, $k_m$ a suitable index that depends on $m$ and $E_i$ are Borel subsets of $[0,1]^n$.
\end{lemma}
\begin{proof}
The result is straightforward by direct inspection of the standard argument used in the case of positive-valued Borel functions. We will sketch such an argument for completeness.   For any $m\ge 1$, we divide the interval $[0,1]$ in $2^m$ subintervals of the type $I_{m_k}=\left[ \frac{k}{2^m}, \frac{k+1}{2^m} \right)$, with $k=0,\dots, 2^m-1$. For any $m$ and $k$, $E_{m_k}=f^{-1}(I_{m_k})$ is a Borel subset of $[0,1]^n$. We then consider the following simple functions:
\begin{center}
$f_m=\sum_{k=0}^{2^m-1}\frac{k}{2^m}\chi_{E_{m_k}}$.
\end{center}
For any choice of indexes, $\frac{k}{2^m}\in [0,1]$ and the sequence is increasing by construction. Moreover, $f_m(x)=\frac{k}{2^m}$ if, and only if, $x\in E_{m_k}$ if, and only if, $\frac{k}{2^m}\le f(x)\le \frac{k+1}{2^m}$. Thus, for any $x\in [0,1]^n$, $\mid f_m(x)-f(x)\mid \le \frac{1}{2^m}$ for any $m$ and the claim is settled.
\end{proof}

\begin{theorem}\label{teo:tribe=Bo}
An $n$-ary function $f:[0,1]^n\to [0,1]$ belongs to the Riesz tribe generated by the projection functions if, and only if, it is Borel measurable.
\end{theorem}
\begin{proof}
Since projection functions are Borel-measurable and $\text{Borel}(n)$ is closed under pointwise suprema, we easily deduce that $\mathcal{RT}_n\subseteq \text{Borel}(n)$.

To prove the other inclusion, first notice that $\mathcal{RT}_n$ is a  $\sigma$-complete Riesz MV-algebra and therefore it is norm complete, that is, any uniform limit of elements of $\mathcal{RT}_n$ belongs to $\mathcal{RT}_n$. Thus, by \Cref{lem:Borel=supSimple}, we only need to prove that the characteristic functions of Borel subsets of $[0,1]^n$ are elements of $\mathcal{RT}_n$.

For $n=1$, it is well known that the Borel subsets of $[0,1]$ are generated by intervals of the type $(r,1]$, with $r>0$. Thus, it is enough to prove the result for functions of the type $\chi_{(r,1]}$. In this case, we write $\chi_{(r,1]}=\bigwedge _m f_{m,r}$, where $f_{m,r}$ is the continuous piecewise linear function with real coefficients defined by
$$
f_{m,r}(x)=
\begin{cases}
0& \text{if } x\le r-\frac{r}{2^m}\\
\text{linear}& \text{if } r-\frac{r}{2^m}<x\le r\\
1& \text{if } x> r
\end{cases}
$$

We recall that \cite[Corollary 7]{LeuRMV} entails that piecewise linear functions with real coefficients belong to the Riesz MV-algebra generated by the projection function.

To extend the argument to the $n$-valued case, it is enough to remark that the Borel subsets of $[0,1]^n$ are generated by multi-intervals of the type $E= \Pi_{i=1}^nE_i$, with $E_i=(r_i,s_i)$ and that $\chi_E=\bigwedge_i \chi_{E_i}$, which we proved to be elements of the Riesz tribe generated by the projection functions.
\end{proof}

Note that the proof of \Cref{teo:tribe=Bo} is a significant simplification of \cite[Theorem 3.3]{DiNN}, where the authors characterize the smallest MV-tribe that contains all projection functions.

\begin{remark}
A different proof of the isomorphism between $IRL_n$ and $\text{Baire}(n)$ can we obtained by \cite[Lemma 1.1]{Coq} and  \Cref{Gamma-DedCompl}.  
\end{remark}

\subsection{Conclusion: logic and analysis in literature}
We conclude this paper by recalling other approaches to the connections between logic and analysis in literature. We start by recalling that some basic connections between the logic of Riesz MV-algebras and functional analysis were made in \cite{DiNLL-RMV}.

One of the first reference to mention is \cite{BenYacoov}, where it is studied a logical system that has metric spaces as models. Such a logic is called \emph{continuous first-order logic} (CFO). It is a predicate logic, whose propositional fragment is axiomatized by the connectives of \L ukasiewicz logic plus an operator that multiplies by $\frac{1}{2}$. The resemblance of such fragment with the logic of Riesz MV-algebras is very clear and it is worth to explore the relations between CFO and the infinitary logic we propose in Section \ref{sec:logic}. Related to this subject, it is worth mention \cite{CI}, where the authors prove that for continuous metric structures, CFO has the same expressive power than Rational Pavelka logic, which is a conservative extension of \L ukasiewicz logic. See also \cite[Historical Remarks]{Iovino}. Moreover, in \cite{BY-I} the authors deal with infinitary continuous logic.

In a recent preprint \cite{abbadini}, M. Abbadini provides a finite axiomatization for several infinitary varieties of lattice-ordered structures. In particular, a finite axiomatization is obtained for Dedekind $\sigma$-complete Riesz Spaces with a weak unit.

Finally, we mention once again \cite{EnzoReggio}, where the authors define a class of MV-algebras enriched with an infinitary operator, for which they provide a finite axiomatization. Such a class, as a category of infinitary algebras, is dual to the category of compact Hausdorff spaces.

\subsection*{Acknowledgment}
\noindent I. Leu\c stean was supported by a grant of the Romanian National Authority for Scientific Research and Innovation, CNCS-UEFISCDI, project number PN-II-RU-TE-2014-4-0730. The authors are deeply grateful to Vincenzo Marra, whose comments and suggestions helped us prove Theorem \ref{inf-hsp}, and to the anonymous referees whose suggestions improved the quality of our presentation.


\begin{thebibliography}{}
\bibitem{abbadini}
Abbadini M., \textit{Dedekind $\sigma$-complete $\ell$-groups and Riesz spaces are varieties}, arXiv:1712.04374 [math.LO].


\bibitem{Alip}
Aliprantis C.D., Border K.C., \textit{Infinite dimensional analysis. A Hitchhiker's Guide}, 3rd edition, Springer (1999).


\bibitem{BGL}
Ball R.N., Georgescu G., Leu\c stean I., \textit{Cauchy completions of MV-algebras}, Algebra Universalis 47(4) (2002) 367-407.

\bibitem{BY-I}
Ben Yacoov I., Iovino J., \textit{Model theoretic forcing in analysis}, Annals of Pure and Applied Logic 158(3) (2009) 163--174.


\bibitem{BenYacoov}
Ben Yacoov I., Usvyatsov A., \textit{Continuous first order logic and local stability}, Transactions of the American Mathematical Society 362 10 (2010) 5213--5259.




\bibitem{Chapter1-HB}
B\v ehounek L., Cintula P., H\'ajek P., \textit{Introduction to mathematical fuzzy logic}, Handbook of Mathematical Fuzzy Logic, Chapter 1, Vol I (Cintula et al., eds). Studies in Logic 37, College Publications, London, 2011.


\bibitem{CI}
Caicedo X., Iovino J., \textit{Omitting uncountable types and the strength of [0,1]-valued logics}, Annals of Pure and Appled Logic 165 1169--1200 (2014).


\bibitem{CDM}
Cignoli R.L.O., D'Ottaviano I.M.L., Mundici D., {\em Algebraic foundations of many-valued reasoning}, Kluwer Academic Publishers, Dordrecht, 2000.


\bibitem{Coq}
Coquand T., \textit{A note on measures with values in a partially ordered vector space}, Positivity 8 (2004) 395--400.


\bibitem{DiNLL-RMV}
Di Nola A., Lapenta S., Leu\c stean I., \textit{An analysis of the logic of Riesz Spaces with strong unit}, Annals of Pure and Applied Logic 169(3) (2018) 216--234.


\bibitem{LeuRMV}
Di Nola A., Leu\c stean I., \textit{\L ukasiewicz logic and Riesz Spaces}, Soft Computing 18(12) (2014) 2349-2363.

\bibitem{DinLeu}
Di Nola A., Leu\c{s}tean I., \textit{\L ukasiewicz Logic and MV-algebras}, Handbook of Mathematical Fuzzy Logic, Vol II (Cintula et al., eds). Studies in Logic 37, College Publications, London, 2011.

\bibitem{DiNN}
Di Nola A., Navara M., \emph{The $\sigma$-complete MV-algebras which have enough states}, Colloq. Math. 103 (2005) 121--130.

\bibitem{Dvu-LS}
Dvure\v censkij A., \textit{Loomis-Sikorski theorem for $\sigma$-complete MV-algebras and $\ell$-groups}, Journal of the Australian Mathematical Society, Series A. Pure Mathematics and Statistics 68(2) (2000) 261-277.

\bibitem{brunella}
Gerla B., \text{Rational \L ukasiewicz logic and Divisible MV-algebras}, Neural Networks World 10(11) (2001), 579-584.

\bibitem{goodearl}
Goodearl K.R., \textit{Partially-ordered abelian groups with interpolation}, Mathematical Surveys and Monographs 20 (2000).


\bibitem{Iovino}
Iovino J., \textit{Applications of Model Theory to Functional Analysis}, Dover Publications (2014).


\bibitem{kakutani}
Kakutani S., \textit{Concrete Representation of Abstract (m)-spaces (A Characterization of the Space of Continuous Functions)} Annals of Mathematics 42(4) (1941) 994-–1024.

\bibitem{Karp}
Karp C.R., \textit{Languages with expressions of infinite length}, North–Holland Publishing Amsterdam (1964).

\bibitem{Kuratowski}
Kuratowski K., \textit{Topology}, vol. 1, Academic Press, New York (1966). 

\bibitem{LLDMV}
Lapenta S., Leu\c stean I., \textit{Notes on DMV-algebras}, Soft Computing  21(21) (2017) 6213--6223.


\bibitem{RS}
Luxemburg W.A.J., Zaanen A.C., \textit{Riesz Spaces I}, North-Holland, Amsterdam (1971).

\bibitem{banachlattices}
Meyer-Nieberg P., \textit{Banach Lattices}, Universitext, Springer (1991).

\bibitem{EnzoReggio}
Marra V., Reggio R., \textit{Stone duality above dimension zero: axiomatising the algebraic theory of C(X)}, Advances in Mathematics 307(5) (2017) 253--287.

\bibitem{Mun1}
Mundici D., \textit{Interpretation of AF C*-algebras in \L ukasiewicz sentential calculus}, J. Funct. Anal. 65 (1986) 15-63.

\bibitem{Mun}
Mundici D., \textit{Tensor products and the Loomis-Sikorski theorem for MV-algebras}, Advanced in Applied Mathematics 22 (1999) 227-248.

\bibitem{MunBook}
Mundici D., \textit{Advanced \L ukasiewicz calculus and MV-algebras}, Trends in Logic 35 Springer (2011).


\bibitem{ScottTarski}
Scott D., Tarski A., \textit{The sentential calculus with infinitely long expressions}, Colloquium Mathematicae 6 (1) (1958) 165-170.

\bibitem{Slomiski}
S\l omi\'{n}ski J., \textit{The theory of abstract algebras with infinitary operations}, Instytut Matematyczny Polskiej Akademi Nauk, Warszawa (1959).

\bibitem{RSZan}
Zaanen A.C., \textit{Riesz Spaces II}, North-Holland, Amsterdam (1983).
\end{thebibliography}
\end{document}